\documentclass[12pt,a4paper,reqno]{amsart}
\usepackage[utf8]{inputenc}
\usepackage[english]{babel}
\usepackage{enumerate}
\usepackage{mathtools}
\usepackage{mathrsfs}
\usepackage{soul, xcolor}
\usepackage{wasysym}
\usepackage{amsmath,amssymb,xpatch,relsize,graphics,graphicx,amsthm}
\usepackage{float,graphicx}
\usepackage{subcaption}

\xapptocmd\normalsize{%
	\abovedisplayskip=12pt plus 3pt minus 9pt
	\abovedisplayshortskip=0pt plus 3pt
	\belowdisplayskip=12pt plus 3pt minus 9pt
	\belowdisplayshortskip=7pt plus 3pt minus 4pt
}{}{}
\theoremstyle{definition}
\newtheorem{definition}{Definition}[section]
\newtheorem{remark}[definition]{Remark}

\theoremstyle{plain}
\newtheorem{theorem}[definition]{Theorem}

\newtheorem{lemma}[definition]{Lemma}

\numberwithin{equation}{section}

\title[Sharp Bohr-Type  inequalities]{Sharp Bohr-Type  inequalities for certain classes of close-to-convex  functions}
\author[S. Rana]{Shalini Rana}
\address{Department of Mathematics, University of Delhi, Delhi--110 007, India}
\email{shalinirana3010@gmail.com}
\author[N.K. Jain]{Naveen Kumar Jain*}
\address {Department of Mathematics, Aryabhatta College, Delhi-110021,India}
\email{naveenjain@aryabhattacollege.ac.in{*}}

\date{}



\allowdisplaybreaks

\begin{document}

\begin{abstract}
In this article, we determine the Rogosinski radii for certain subclasses of close-to-convex functions defined on open unit disc $\mathbb{D}= \{z \in \mathbb{C}: |z| < 1\}$. Furthermore, we establish improved versions of the classical Bohr inequality and the Bohr-Rogosinski inequality pertaining to these subclasses. We demonstrate that all results derived in the study are sharp.
\end{abstract}
\maketitle

\section{Introduction and preliminaries}
Let $\mathcal{A}$ be the class of analytic functions $g$ with the Taylor series expansion $g(z)=\sum_{n=0}^{\infty}a_nz^n$ in the unit disc $\mathbb{D}=\{ z \in \mathbb{C}: |z|<1 \}$. The classical Bohr inequality \cite{boh1914} says that if $g\in\mathcal{A}$  and  $g$ maps the unit disc into itself, then
\begin{equation}\label{first}
\sum_{n=0}^{\infty}|a_n|r^n\leq1\quad\text{for}\quad |z|=r\leq\frac{1}{3}
\end{equation}
and the radius  $1/3$ cannot be improved. The radius $1/3$ is known as Bohr radius, while the inequality \eqref{first} is known as Bohr inequality. Initially, the inequality \eqref{first} was established by Harald Bohr for radius $r\leq1/6$ and later Weiner, Riesz and Schur obtained the inequality for $r\leq1/3$. It is easy to see that $(1-|a_0|)$ is equal to $d(g(0), \partial g(\mathbb{D}))$, where $d$ is the Euclidean distance and $\partial g(\mathbb{D})$  is the boundary of $g(\mathbb{D})$. Therefore, the Bohr inequality can be re-written in the following Euclidean distance form
\begin{equation}\label{second}
d\left(\sum_{n=0}^{\infty}|a_nz^n|, |a_0|\right)=\sum_{n=1}^{\infty}|a_nz^n|\leq1-|a_0|=d(g(0), \partial g(\mathbb{D})).
\end{equation}
The inequality \eqref{second} can be generalized to class $\mathcal{G}$ consisting of functions $g$ analytic in $\mathbb{D}$ such that $g(\mathbb{D})\subset\Delta$, where $\Delta$ is any given domain. The class $\mathcal{G}$ satisfies the Bohr phenomenon if there exists largest $r_{\Delta}\in (0, 1)$ such that
\begin{equation*}
d\left(\sum_{n=0}^{\infty}|a_nz^n|, |a_0|\right)=\sum_{n=1}^{\infty}|a_nz^n|\leq d(g(0), \partial \Delta)
\end{equation*}
holds for all $|z|=r\leq r_{\Delta}$ and for all the functions $g\in\mathcal{G}$. It is interestingly
enough, for the convex domain $\Delta$, the radius $r_{\Delta}=1/3$ satisfying \eqref{second} coincides with Bohr radius  \cite{aiz2007}. The Bohr radius has been obtained for the class $\mathcal{G}$ when $\Delta$ has positive real part \cite{Ali2019} and  Janowski starlike \cite{Ana2021} (see also \cite{Ahu2020}). Several  Bohr radii have been obtained for the classes of analytic functions $g$ satisfying the differential
subordination $g(z)+\beta zg'(z)+\gamma z^2g''(z)\prec k(z)$, when the function $k$ is convex and starlike \cite{Abu2014}, convex and starlike of order $\alpha$ \cite{Jai2020} and  Janowski starlike \cite{Ana2021}.
\\
\noindent
The Bohr inequality has emerged as an active research area for operator algebraists after Dixon \cite{Dix1995} established a connection between the inequality and the characterization of Banach Algebra that satisfies Von Neumann inequality. Improved versions of Bohr inequality for various classes of analytic functions  becomes now a days an active research area. Kayumov et al. \cite{Kay2018} obtained several different improved versions of the classical Bohr inequality. Recently, Ahamed and  Ahammed \cite{ahm2023} obtain sharp refined Bohr–Rogosinski-type
inequalities for the subordination class of univalent and convex function. In 2021, Liu et al. \cite{LIU2021} generalized and improved several Bohr-type inequalities for bounded analytic functions. Recently, Rana and Jain  \cite{sha2025} has established several improved Bohr-type inequalities for starlike functions of order $\alpha$,  functions whose derivative has positive real part and the  functions  convex in the direction of imaginary axis, see also \cite{Pon2020, Hua2020, Liu2023, Pon2023}.
\\
\noindent
In 1952, Kaplan \cite{Kap1952} introduced the class of close-to-convex functions $f\in\mathcal{A}$ satisfying the condition
\begin{equation}\label{eq23}
Re \left(\frac{zf^{'}}{g} \right) >0\quad (z \in \mathbb{D}),
\end{equation}
where $g$ is a function starlike  with respect to origin  of the form
\begin{equation}\label{eq24}
g(z)=b_1z+...  Re(b_1)> 0.
\end{equation}
The class is denoted by $\mathcal{C}$. Later, Silverman and Telage \cite{sil1979} introduced the three subclasses, $\mathcal{C}_1$, $\mathcal{C}_2$ and $\mathcal{C}_3$ of class $\mathcal{C}$.
\noindent
A function $f$ is said to be in the class $\mathcal{C}_1$ if there exists a convex function $g$ of the form \eqref{eq24} such that \eqref{eq23} is satisfied.
A function $f$ is said to be in the class $\mathcal{C}_2$ if there exists a convex function $g$ of the form \eqref{eq24} satisfying
\begin{equation*}
Re\left( \frac{(zf^{'})^{'}}{g^{'}}  \right)>0\quad (z \in \mathbb{D}).
\end{equation*}
\noindent
A function $f$ is said to be in the class $\mathcal{C}_3$ if
\begin{equation*}
Re\left(   \frac{z((zf^{'})^{'})^{'}}{(zg^{'})^{'}}  \right)>0 \quad (z \in \mathbb{D}),
\end{equation*}
where  $g$ is a convex function of the form \eqref{eq24}. They proved that  $\mathcal{C}_3 \subset \mathcal{C}_2\subset \mathcal{C}_1 \subset \mathcal{C}$.
In 2017,
Kayumov and Ponnusamy \cite{kay2017} introduced and found the following Bohr-Rogosinski inequality and  Bohr–Rogosinski radius for the class of  functions mapping unit disc to unit disc.
\begin{theorem}\cite{kay2017}
Suppose that $f:\mathbb{D}\rightarrow\mathbb{D}$ and $f(z)=\sum_{n=0}^{\infty}a_nz^n$. Then
\begin{equation*}
|f(z)|+\sum_{n=N}^{\infty}|a_n r^n|\leq1 \quad \text{for}\, |z|=r\leq R_N,
\end{equation*}
where $R_N$ is the positive root of the equation $2(1+r)r^N-(1-r)^2=0$. The radius $R_N$ is the best possible. Moreover,
\begin{equation*}
|f(z)|^2+\sum_{n=N}^{\infty}|a_n r^n|\leq1 \quad \text{for} \,|z|=r\leq R'_N,
\end{equation*}
where $R'_N$ is the positive root of the equation $(1+r)r^N-(1-r)^2=0$. The radius $R'_N$ is the best possible.
\end{theorem}
\noindent
For the results on Bohr-Rogosinski inequality, we refer to the articles \cite{ALK2020, KAV2021, PON2022,Pon2020}.
\\
\noindent
In the present work, we extend this line of research by determining the sharp Bohr–Rogosinski radius for the classes $\mathcal{C}_1, \mathcal{C}_2$ and $\mathcal{C}_3$ together with  improved
versions of the classical Bohr’s inequality . In section 2, section 3, and section 4, we find the sharp Bohr–Rogosinski radius and its various versions for the classes $\mathcal{C}_1, \mathcal{C}_2$ and $\mathcal{C}_3$ respectively.

\begin{section}{Bohr-type inequalities for the class $\mathcal{C}_1$}
\noindent
We need the following lemmas for main results.
\begin{lemma}\label{lm1}\cite{Sil1972}
If $f(z)= z + \sum_{n=2}^{\infty} a_n z^n \in \mathcal{C}_1$, then
\begin{equation} \label{coef1}
|a_n| \le 2- \frac{1}{n}
\end{equation}
with equality for $k(z,1,-1)$, where
\begin{equation*}
k(z,x,y)= (1+x) \frac{z}{(1-yz)} + x \bar{y} \log{(1-yz)}
\end{equation*}
for $z \in \Delta$ and $|x|=|y|=1$.
\end{lemma}
\begin{lemma}\label{lm2}{\cite{Sil1972}}
If $f \in \mathcal{C}_1$ then
\begin{equation} \label{growth1}
\frac{2r}{1+r} -\log(1+r) \le |f(z)| \le \frac{2r}{1-r} +\log(1-r)
\end{equation}
and
\begin{equation} \label{distortion1}
\frac{1-r}{(1+r)^2} \le |f^{'}(z)| \le \frac{1+r}{(1-r)^2}
\end{equation}
for all $|z| \le r$ and equality holds for $k(z,1,1)$ at $z=\pm r$  where
\begin{equation*}
k(z,x,y)= (1+x) \frac{z}{(1-yz)} + x \bar{y} \log{(1-yz)}
\end{equation*}
for $z \in \Delta$ and $|x|=|y|=1$.
\end{lemma}
\noindent
We begin by determining the sharp Bohr radius associated with the Bohr inequality for analytic functions $f(z)$. In this modified form of the inequality, the classical coefficients $|a_0|$ and $|a_1|$ are replaced by $|f(z)|$ and $|f^{'}(z)|$, respectively, allowing us to explore a functional analogue of the traditional Bohr-type estimate.
\begin{theorem}
Suppose that $f(z)=z + \sum_{n=2}^{\infty} a_n z^n \in \mathcal{C}_1$, then
\begin{equation}\label{2.1.1}
|f(z)| + |f^{'}(z)||z| + \sum_{n=2}^{\infty} |a_n z^n| \le d( f(0), \partial{f(\mathbb{D})})
\end{equation}
for $|z| \le r_{11}$, where $r_{11}=0.110377$ is the solution of
\begin{equation*}
    \log2-1-r (-6+r (2+r-\log2)+\log4)+2 (1-r)^2 \log(1-r) =0.
\end{equation*}
The result is sharp.
\end{theorem}
\begin{proof}
We begin the proof by evaluating $d(f(0), \partial{f(\Delta)})$.
From growth estimate of $f(z) \in \mathcal{C}_1$ in equation ($\ref{growth1}$), we get
\begin{equation*}
d(f(0), \partial{f(\mathbb{D})}) = \liminf \limits_{|z| \rightarrow 1} |f(z)-f(0)| \ge \lim \limits_{r \rightarrow 1} \left(\frac{2r}{1+r}- \log(1+r)\right)
\end{equation*}
which gives
\begin{equation}\label{eq1}
d(f(0), \partial{f(\mathbb{D})}) \ge (1-\log{2}).
\end{equation}
For $|z|\leq r$, by $(\ref{coef1})$, $(\ref{growth1})$, $(\ref{distortion1})$ , we have
\begin{equation}\label{eq3}
|f(z)| + |f^{'}(z)||z| + \sum_{n=2}^{\infty} |a_n| |z|^n\leq \frac{2r}{1-r} + \log(1-r) + \frac{r(1+r)}{(1-r)^2} + \sum_{n=2}^{\infty} \left( 2-\frac{1}{n} \right) r^{n}
\end{equation}
Define a function $P:[0, 1)\rightarrow \mathbb{R}$ by
\begin{equation*}
P(r)=  \log2-1-r (-6+r (2+r-\log2)+\log4)+2 (1-r)^2 \log(1-r).
\end{equation*}
It can be easily seen that $P(0)=\log2-1 \simeq-0.306853< 0$ and $P(1/2)={11}/{8}-{\log2}/{4}\simeq1.20171 > 0$. Therefore by intermediate value theorem $P(r)$ has  a  root in $ (0,1/2)$. Next to show  the uniqueness of root in $(0, 1/2)$, it is sufficient to show that the function $P$ is a strictly increasing function in $(0,1/2)$.
 It is easy to see that  $P'(r)=(2+6 r-5 r^2+r^3)/{(1-r)^3}>(2+6 r-5 r+r^3)/{(1-r)^3}=(2+ r+r^3)/{(1-r)^3}>0$ for all $r \in (0,1)$, which shows that $P$ is strictly increasing function.  Therefore, $P(r)$ has the unique root in $(0, 1/2)$, say $r_{11}=0.110377$.
  Thus,
  \begin{equation}\label{eq4}
  P(r_{11})=  \log2-1-r_{11} (-6+r_{11} (2+r_{11}-\log2)+\log4)+2 (1-r_{11})^2 \log(1-r_{11})=0
  \end{equation}
   and $P(r)\leq0$ for $r\in(0, r_{11})$.
\noindent
 In view of the above argument together with   \eqref{eq1}   and \eqref{eq3}, for  $|z| \le r_{11}$, we have
\begin{align*}\label{eq2}
|f(z)| + &|f^{'}(z)||z| + \sum_{n=2}^{\infty} |a_n| |z|^n-d(f(0), \partial{f(\mathbb{D})})\\
& \le \frac{2r}{1-r} + \log(1-r) + \frac{r(1+r)}{(1-r)^2} + \sum_{n=2}^{\infty} \left( 2-\frac{1}{n} \right) r^{n}-(1-\log2)\\
&= \frac{\log2-1-r (-6+r (2+r-\log2)+\log4)+2 (1-r)^2 \log(1-r)}{(1 - r)^2}\\
&\leq0
\end{align*}
\begin{figure}
  \centering
  \includegraphics[width=8cm]{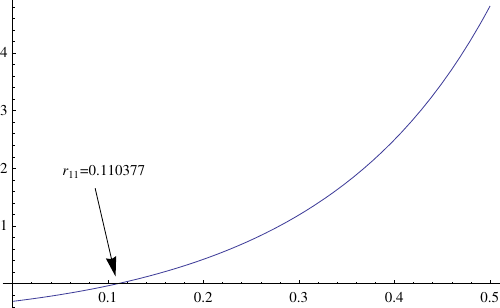}\\
  \caption{The radius $r_{11}=0.110377$ is sharp}
\end{figure}
\noindent
To prove the sharpness, we assume
\begin{equation*}
    f(z)=\frac{2z}{(1+z)}- \log(1+z).
\end{equation*}
Note that the power series expansion of $f(z)$ is given by
\begin{equation*}
    f(z)= \sum_{n=1}^{\infty} (-1)^{n+1} \left(2-\frac{1}{n} \right) z^n.
\end{equation*}
At $z=-r=-r_{11}$, by \eqref{eq4}, we have
\begin{align*}
|f(z)| + |f^{'}(z)||z| + &\sum_{n=2}^{\infty} |a_n| |z|^n\\
&= \left|\frac{2z}{(1+z)}- \log(1+z) \right| +  \left| \frac{(1-z)z}{(1+z)^2} \right|  +   \sum_{n=2}^{\infty} \left| (-1)^{n+1} \left(2-\frac{1}{n} \right)z^n \right| \\
& =  \left|\frac{2(-r)}{(1-r)}- \log(1-r) \right| + \left| \frac{-r(1+r)}{(1-r)^2} \right|  +   \sum_{n=2}^{\infty}  \left|2-\frac{1}{n} \right| |-r|^n\\
&=  \frac{2r}{(1-r)}+\log(1-r) +  \frac{r(1+r)}{(1-r)^2}   +   \sum_{n=2}^{\infty}  \left(2-\frac{1}{n} \right) |r|^n\\
&=1-\log 2.\\
&=\liminf \limits_{|z| \rightarrow 1} \left(\frac{2z}{1+z}- \log(1+z)\right)\\
&=d(f(0), \partial{f(\mathbb{D})}).
\end{align*}
Hence, the result is sharp.
\end{proof}
\noindent
Next, by considering powers of the coefficient $(|a_n|)$, we evaluate the corresponding sharp Bohr radii for the Bohr inequality for class $\mathcal{C}_1$. This extension allows us to analyze how the powers of the moduli of the coefficients influence the Bohr-type bounds, thereby enriching the understanding of Bohr phenomena in the class $\mathcal{C}_1$.
\begin{theorem}
Suppose that $f(z)=z + \sum_{n=2}^{\infty} a_n z^n \in \mathcal{C}_1$, then
\begin{equation} \label{2.1.2}
|z|+ \sum_{n=2}^{\infty} |a_n z^n| + \sum_{n=2}^{\infty} |a_n|^{p} |z|^{np}  \le d(f(0), \partial{f(\mathbb{D})})
\end{equation}
for $|z| \le r_p$ and $p \ge 1$, where $r_p$ is the solution of
\begin{equation*}
(1-r)\sum_{n=2}^{\infty} \left( 2-\frac{1}{n} \right)^{p} r^{pn}- (1-3 r+r \log2-\log(2-2 r)+r \log(1-r))=0.
\end{equation*}
The result is sharp.
\end{theorem}
\begin{proof}
By Lemma \ref{lm2}, it follows that
\begin{equation}\label{eq5}
d(f(0), \partial{f(\mathbb{D})})= \liminf \limits_{|z| \rightarrow 1} |f(z)-f(0)|
 \ge \lim \limits_{r \rightarrow 1} \left(\frac{2r}{1+r}- \log(1+r)\right)
\end{equation}
which gives $d(f(0), \partial{f(\mathbb{D})}) \ge (1-\log{2})$ for $f \in \mathcal{C}_1$. By Lemma \ref{lm1},  for $|z| \le r$ and $p \ge 1$, we have
\begin{align*}
|z|+ \sum_{n=2}^{\infty} |a_n| |z|^n&+ \sum_{n=2}^{\infty} |a_n|^{p} |z|^{np} \le r+ \sum_{n=2}^{\infty} \left( 2-\frac{1}{n} \right) r^{n}+ \sum_{n=2}^{\infty} \left( 2-\frac{1}{n} \right)^{p} r^{pn}\\
&=\sum_{n=2}^{\infty} \left( 2-\frac{1}{n} \right)^{p} r^{pn}- \frac{(1-3 r+r \log2-\log(2-2 r)+r \log(1-r))}{1-r}.
\end{align*}
Let  $Q$ be a function  defined by
\begin{align*}
Q(r)=\sum_{n=2}^{\infty} \left( 2-\frac{1}{n} \right)^{p} r^{pn}- \frac{1-3 r+r \log2-\log(2-2 r)+r \log(1-r)}{1-r}, \quad r\in [0, 1).
\end{align*}
Since, $Q(0)=-(1-\log 2)<0$ and
 \[
 Q(1/2)
 = 1+ \sum_{n=2}^{\infty} \left( 2-\frac{1}{n} \right)^{p} \frac{1}{2^{pn}}> 0
\]
by the Intermediate Value Theorem it follows that $Q(r)=0$ has a root in $(0, 1/2)$.
The uniqueness of the root $r$ in $(0, 1/2)$ follows from the fact that the derivative of the  function \[Q'(r)=1+ \sum_{n=2}^{\infty} \left( 2-\frac{1}{n} \right)n r^{n-1}+ \sum_{n=2}^{\infty} \left( 2-\frac{1}{n} \right)^{p}pn r^{p n-1}>0.\]
 Thus,
\begin{equation*}
\frac{2 r+(1-r) \log(1-r)}{1-r}+\sum_{n=2}^{\infty} \left( 2-\frac{1}{n} \right)^{p} r^{pn} \le  (1 - \log{2}).
\end{equation*}
 for $r \le r_p$ and $p \ge 1$, where $r_p$ is the solution of the equation $Q(r)=0$ in $(0, 1).$
To prove the sharpness, we assume
\begin{equation}\label{eq7}
    f(z)=\frac{2z}{(1+z)}- \log(1+z)= \sum_{n=1}^{\infty} (-1)^{n+1} \left(2-\frac{1}{n} \right) z^n.
\end{equation}
For  $z=r=r_p$, we have
\begin{align*}
|z|+ \sum_{n=2}^{\infty} |a_n z^n|& + \sum_{n=2}^{\infty} |a_n|^{p} |z|^{np}\\
&=|z|+ \sum_{n=2}^{\infty} \left| (-1)^{n+1} \left(2-\frac{1}{n} \right) z^n \right| + \sum_{n=2}^{\infty} {\left| (-1)^{n+1} \left(2-\frac{1}{n} \right) \right| }^p |z|^{np}\\
& =   |r|+ \sum_{n=2}^{\infty} \left| \left(2-\frac{1}{n} \right) r^n \right| + \sum_{n=2}^{\infty} {\left| \left(2-\frac{1}{n} \right) \right| }^p |r|^{np}\\
&= r+ \sum_{n=2}^{\infty}  \left(2-\frac{1}{n} \right) r^n  + \sum_{n=2}^{\infty} { \left(2-\frac{1}{n} \right)  }^p r^{np}\\
&=\frac{2 r+(1-r) \log(1-r)}{1-r}+\sum_{n=2}^{\infty} { \left(2-\frac{1}{n} \right)  }^p r^{np}\\
&=1-\log 2\\
& =\lim \limits_{r \rightarrow 1} \left(\frac{2r}{1+r}- \log(1+r)\right)\\
&= d(f(0), \partial{f(\mathbb{D})})
\end{align*}
Hence, the result is sharp.
\end{proof}

\begin{remark}
For different values of $p$ the radius converges to $0.215585$ upto $5$ decimal places as evident from the table given below
\end{remark}
\begin{table}[h!]
	\begin{center}
		\begin{tabular}{ |c|c|c| }
			\hline
			S. No. & p & $r_p$\\
			\hline
			1 & 2 & 0.213087\\
			\hline
			2 & 3 & 0.215411 \\
			\hline
			3 & 4 & 0.215573\\
			\hline
			4 & 5 & 0.215584\\
			\hline
			5 & 6& 0.215584\\
			\hline
			6 & 7& 0.215585\\
			\hline
			7& 8& 0.215585\\
			\hline
		\end{tabular}
		\caption{\textbf{Calculation of the  radius $r_p$ for various values of $p$.}}
		\label{Table:1}
	\end{center}
\end{table}
\noindent
In the  next two theorems \ref{th3} and \ref{th4}, we find the  Rogosinski's inequality and Rogosinski's radius for the class $\mathcal{C}_1$.
\begin{theorem}\label{th3}
Suppose that $f(z)=z + \sum_{n=2}^{\infty} a_n z^n \in \mathcal{C}_1$ and $N\geq2$ is an integer, then
\begin{equation} \label{2.1.3}
|f(z)|+ \sum_{n=N}^{\infty} |a_n z^n|  \le d(f(0), \partial{f(\mathbb{D})})
\end{equation}
for $|z| \le r_{1,N}$, where $r_{1,N}$ is the solution of
\begin{equation*}
\frac{2r}{1-r} +\log(1-r) + \sum_{n=N}^{\infty} \left( 2-\frac{1}{n} \right) r^{n}- (1 - \log{2})=0.
\end{equation*}
The result is sharp.
\end{theorem}
\begin{proof} Let $f\in  \mathcal{C}_1$. Using
  coefficients bound  ($\ref{coef1}$) and growth estimate ($\ref{growth1}$), for $|z| \le r$, we have
\begin{equation}\label{eq6}
|f(z)|+ \sum_{n=N}^{\infty} |a_n z^n|  \le
\frac{2r}{1-r} +\log(1-r) + \sum_{n=N}^{\infty} \left( 2-\frac{1}{n} \right) r^{n}.
\end{equation}
By \eqref{eq5} and \eqref{eq6}, the  inequality  holds
\begin{equation*}
|f(z)|+ \sum_{n=N}^{\infty} |a_n z^n|  \le  d(f(0), \partial{f(\mathbb{D})})
\end{equation*}
if and only if \begin{equation*}
\frac{2r}{1-r} +\log(1-r) + \sum_{n=N}^{\infty} \left( 2-\frac{1}{n} \right) r^{n} \le (1 - \log{2}),\quad |z|\leq r.
\end{equation*}
Let the function $R:[0, 1)\rightarrow \mathbb{R}$ be defined by
\begin{equation*}
R(r)=\frac{2r}{1-r} +\log(1-r) + \sum_{n=N}^{\infty} \left( 2-\frac{1}{n} \right) r^{n}- (1 - \log{2}).
\end{equation*}
Note that $R(0)=\log2-1<0$ and \[R\left(\frac{1}{2}\right)=2+\sum_{n=N}^{\infty} \left( 2-\frac{1}{n} \right) r^{n}- (1 - \log{2})>0.\]
Using the Intermediate Value Theorem, there exists a root $r\in(0, 1/2)$  of the equation $R(r)=0$. Also,
\begin{align*}
R'(r)&=\frac{2}{(1-r)^2}-\frac{1}{1-r}+\frac{2 r^{-1+N} (N+r-N r)}{(-1+r)^2}-\frac{r^{-1+N}}{1-r}\\
&=\frac{r+r^2+(2 N-1) r^N+(3-2 N) r^{1+N}}{(1-r)^2 r}\\
&\geq0.
\end{align*}
Therefore, the equation $R(r)=0$ has a unique root, say $r_{1,N}$. Thus, the  equation $(\ref{2.1.3})$ holds for $r \le r_{1,N}$.
The radius is sharp for the function $f$ defined in $\eqref{eq7}$.%
\end{proof}
\noindent
Next, we state the second result determining the sharp Bohr-Rogosinski radius for the class $\mathcal{C}_1$.
\begin{theorem}\label{th4}
Let $f(z)=z + \sum_{n=2}^{\infty} a_n z^n \in \mathcal{C}_1$ and $N\geq2$ be an integer. Then
\begin{equation} \label{2.1.4}
|f(z)|^2+ \sum_{n=N}^{\infty} |a_n z^n|  \le d(f(0), \partial{f(\mathbb{D})})
\end{equation}
for $|z| \le r_{2,N}$, where $r_{2,N}$ is the solution of
\begin{equation*}
\left(\frac{2r}{1-r} +\log(1-r) \right)^2 + \sum_{n=N}^{\infty} \left( 2-\frac{1}{n} \right) r^{n}- (1 - \log{2})=0.
\end{equation*}
The result is sharp.
\end{theorem}
\begin{proof}
 By Lemma \ref{lm1} and Lemma \ref{lm2}, for $|z| \le r$,
we have
\begin{equation*}
|f(z)|^2+ \sum_{n=N}^{\infty} |a_n z^n|  \le
 \left( \frac{2r}{1-r} +\log(1-r) \right)^2 + \sum_{n=N}^{\infty}  \left( 2-\frac{1}{n} \right) r^{n}.
\end{equation*}
To find the radius $r$ satisfying the inequality
\begin{equation*}
|f(z)|^2+ \sum_{n=N}^{\infty} |a_n z^n|  \le d(f(0), \partial{f(\mathbb{D})})\quad (|z|\leq r),
\end{equation*}
we define the function $S:[0, 1)\rightarrow\mathbb{R}$ such that
\begin{equation*}
S(r)=\left(\frac{2r}{1-r} +\log(1-r) \right)^2 + \sum_{n=N}^{\infty} \left( 2-\frac{1}{n} \right) r^{n}- (1 - \log{2}).
\end{equation*}
We claim that $S(r)\leq 0$ for $r \le r_{2,N}$, where $r_{2,N}$ is the solution of the equation
\begin{equation}\label{eq8}
\left(\frac{2r}{1-r} +\log(1-r) \right)^2 + \sum_{n=N}^{\infty} \left( 2-\frac{1}{n} \right) r^{n}- (1 - \log{2})=0.
\end{equation}
An easy calculation shows that $S(0)=\log2-1<0$ and
\[S\left(\frac{1}{2}\right)=4+\sum_{n=N}^{\infty} \left( 2-\frac{1}{n} \right) \frac{1}{2^{n}}- (1 - \log{2})
\geq4-(1-\log2)
=1+\log2
>0\]
 Clearly, by the Intermediate Value Theorem, there exists a root $r\in(0, 1/2)$  of the equation $S(r)=0$. Note that, for $r\in(0, 1/2)$,
 \begin{align*}S'(r)&=\left(\frac{2}{(1-r)^2}-\frac{1}{1-r}\right)^2+\frac{2 r^{N-1} (N+r-N r)}{(-1+r)^2}-\frac{r^{N-1}}{1-r}\\
 &=\left(\frac{2}{(1-r)^2}-\frac{1}{1-r}\right)^2+\frac{2N r^{N-1}(1-r) }{(1-r)^2}+\frac{2r^N}{(1-r)^2}-\frac{r^{N-1}}{1-r}\\
 &=\left(\frac{2}{(1-r)^2}-\frac{1}{1-r}\right)^2+\frac{2N r^{N-1}(1-r) }{(1-r)^2}+\frac{(3-r)r^{N-1}}{(1-r)^2}\\
 &\geq0
 \end{align*}
It shows that the function $S(r)$ is strictly increasing in $(0, 1/2)$. Thus, $S(r)\leq 0$ for $r \le r_{2,N}$, where $r_{2,N}$ is the unique root of the equation \eqref{eq8}.

To prove the sharpness, we assume
\begin{equation*}
    f(z)=\frac{2z}{(1+z)}- \log(1+z).
\end{equation*}

By \eqref{eq8}, at $z=-r=-r_{2,N}$, we have,
\begin{align*}
|f(z)|^2 + \sum_{n=N}^{\infty} |a_n z^n| & =
\left| \frac{2z}{1+z} -\log(1+z) \right|^2 + \sum_{n=N}^{\infty} \left| (-1)^{n+1}\left( 2-\frac{1}{n} \right) z^{n} \right|\\
&= \left| \frac{-2r}{1-r} -\log(1-r) \right|^2 + \sum_{n=N}^{\infty} \left|\left( 2-\frac{1}{n} \right) (-r)^{n} \right|\\
&= \left( \frac{2r}{1-r} +\log(1-r) \right)^2 + \sum_{n=N}^{\infty} \left( 2-\frac{1}{n} \right) r^{n}\\
&=(1 - \log{2})\\
&=\liminf \limits_{r \rightarrow 1} \left(\frac{2r}{1+r}- \log(1+r)\right)\\
&=d(f(0), \partial{f(\mathbb{D})}).
\end{align*}
Hence, the result is sharp.
\end{proof}

\section{certain Bohr radii for the class $\mathcal{C}_2$}
\noindent
We need the following lemmas for main results.
\begin{lemma} {\cite{Sil1972}}
If $f(z)= z + \sum_{n=2}^{\infty} a_n z^n \in \mathcal{C}_2$, then
\begin{equation}\label{coef2}
 |a_n| \le 1
\end{equation}
\begin{equation}\label{growth2}
\frac{r}{1+r}  \le |f(z)| \le \frac{r}{1-r}
\end{equation}
\begin{equation}\label{distortion2}
\frac{1}{(1+r)^2} \le |f^{'}(z)| \le \frac{1}{(1-r)^2}
\end{equation}
for all $|z| \le r$. Equality in all cases is obtained for $f(z)=z/(1-z)$.
\end{lemma}
\noindent
Using the above coefficient bound, growth and distortion results we obtain the following Bohr-type inequalities for the class $\mathcal{C}_2$.
\begin{theorem}\label{th1}
Suppose that $f(z)=z + \sum_{n=2}^{\infty} a_n z^n \in \mathcal{C}_2$, then
\begin{equation}\label{2.2.1}
|f(z)| + |f^{'}(z)||z| + \sum_{n=2}^{\infty} |a_n z^n| \le d( f(0), \partial{f(\mathbb{D})})
\end{equation}
for $|z| \le r_{21}\simeq 0.173417$, where $r_{21}$ is the solution of
\begin{equation*}
    1-6 r+r^2+2 r^3 =0.
\end{equation*}
The result is sharp.
\end{theorem}
\begin{proof}
We begin the proof by evaluating
\begin{equation*}
d(f(0), \partial{f(\mathbb{D})})= \liminf \limits_{|z| \rightarrow 1} |f(z)-f(0)|.
\end{equation*}
From growth estimate of $f(z) \in \mathcal{C}_2$ in equation ($\ref{growth2}$), we get
\begin{equation*}
d(f(0), \partial{f(\mathbb{D})}) \ge \lim \limits_{r \rightarrow 1} \left(\frac{r}{1+r}\right)
\end{equation*}
which gives
\begin{equation}\label{eq9}
d(f(0), \partial{f(\mathbb{D})}) \ge 1/2.
\end{equation}
\noindent
 Equations $(\ref{coef2})$, $(\ref{growth2})$ and $(\ref{distortion2})$ yield that for all $|z| \le r$, we have
\begin{equation}\label{eq10}
|f(z)| + |f^{'}(z)||z| + \sum_{n=2}^{\infty} |a_n| |z|^n
\le \frac{r}{1-r} + \frac{r}{(1-r)^2} + \sum_{n=2}^{\infty} r^{n}=\frac{r \left(2-r^2\right)}{(1-r)^2}.
\end{equation}
An easy calculation shows that
\begin{align}\label{eq11}
\frac{r \left(2-r^2\right)}{(1-r)^2}&\leq \frac{1}{2}\nonumber\\
\text{if and only if}\quad \frac{1-6 r+r^2+2 r^3}{2 (1-r)^2}&\geq0\nonumber\\
\text{if and only if} \quad1-6 r+r^2+2 r^3&\geq0.
\end{align}
By \eqref{eq10} and \eqref{eq11}, it follows that the inequality
\[|f(z)| + |f^{'}(z)||z| + \sum_{n=2}^{\infty} |a_n z^n| \le d( f(0), \partial{f(\mathbb{D})})\]
holds for all $|z|\leq r$, where $r=r_{21}\simeq0.173417\in(0, 1)$ is the root of the equation $1-6 r+r^2+2 r^3=0.$
\begin{figure}
  \centering
  \includegraphics[width=10cm]{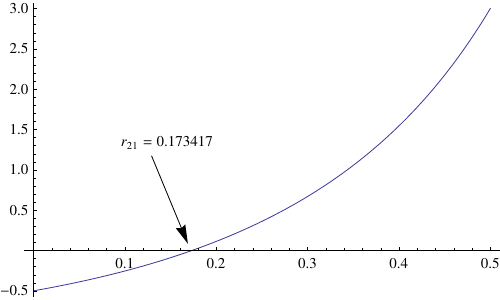}\\
  \caption{The radius $r_{21}=0.173417$ is sharp}
\end{figure}
\\
\noindent
To prove the sharpness, consider the function
\begin{equation*}
    f(z)=\frac{z}{(1-z)}=\sum_{n=1}^{\infty} z^n.
\end{equation*}
\\
\noindent
At $z=r_{21}$, we have
\begin{align*}
|f(z)| + |f^{'}(z)||z| + \sum_{n=2}^{\infty} |a_n| |z|^n &= \left|\frac{z}{(1-z)}\right|+ \left|\frac{z}{(1-z)^2}\right| +  \sum_{n=1}^{\infty}  |z|^n\\
&=\frac{r_{21}}{(1-r_{21})}+ \frac{r_{21}}{(1-r_{21})^2} +  \sum_{n=1}^{\infty}  r_{21}^n\\
&=\frac{r_{21} \left(2-r_{21}^2\right)}{(1-r_{21})^2}\\
&=\frac{1}{2}\\
&=d(f(0), \partial{f(\mathbb{D})}).
\end{align*}
Hence, the result is sharp.
\end{proof}
\noindent
Next, by considering powers of the coefficient $(|a_n|)$, we evaluate the corresponding sharp Bohr radii for the class $\mathcal{C}_2$.
\begin{theorem}
Suppose that $f(z)=z + \sum_{n=2}^{\infty} a_n z^n \in \mathcal{C}_2$, then
\begin{equation} \label{2.2.2}
|z|+ \sum_{n=2}^{\infty} |a_n z^n| + \sum_{n=2}^{\infty} |a_n|^{p} |z|^{np}  \le d(f(0), \partial{f(\mathbb{D})})
\end{equation}
for $|z| \le r^*_{p}$ and $p \ge 1$, where $r^*_{p}$ is the solution of
\begin{equation*}
1-3 r-r^p-2 r^{2 p}+3 r^{1+p}+2 r^{1+2 p}=0.
\end{equation*}
The result is sharp.
\end{theorem}
\begin{proof} Let $f\in\mathcal{C}_2.$
Now proceeding as  in the previous Theorem \ref{th1}, we get
\begin{equation*}
d(f(0), \partial{f(\mathbb{D})}) = \liminf \limits_{|z| \rightarrow 1} |f(z)-f(0)| \ge \lim \limits_{r \rightarrow 1} \left(\frac{r}{1+r}\right) = \frac{1}{2}
\end{equation*}
\\
\noindent
It follows from the equations $(\ref{coef2})$, $(\ref{growth2})$ and $(\ref{distortion2})$ that for all $|z| \le r$, $p \ge 1$,  the following inequality holds
\begin{equation*}
|z|+ \sum_{n=2}^{\infty} |a_n| |z|^n+ \sum_{n=2}^{\infty} |a_n|^{p} |z|^{np} \le r+ \sum_{n=2}^{\infty} r^{n}+ \sum_{n=2}^{\infty} r^{pn}=r+\frac{r^2}{1-r}+\frac{r^{2 p}}{1-r^p}.
\end{equation*}
Also, \[r+\frac{r^2}{1-r}+\frac{r^{2 p}}{1-r^p}=\frac{r+r^{2 p}-r^{1+p}-r^{1+2 p}}{(1-r) \left(1-r^p\right)}=\frac{1}{2}\]
if and only if $1-3 r-r^p-2 r^{2 p}+3 r^{1+p}+2 r^{1+2 p}=0$.
Let us assume that
\begin{equation*}
Q(r)=1-3 r-r^p-2 r^{2 p}+3 r^{1+p}+2 r^{1+2 p}.
\end{equation*}
Since, $Q(0)=1$ and
\noindent
\[Q(1/2)=2^{-1 - 2 p} (-2 + 2^p - 4^p)=-2^{-1 - 2 p} (2 - 2^p + 4^p)=-2^{-1 - 2 p} (2 + 2^p( 2^p-1)<0,\] by Intermediate Value Theorem, $Q(r)$ has a root in $(0, 1/2)$. To prove that the function $Q$ is decreasing  we need to find the maximum of the function
$j(p)=pr^p, (p\geq 1, r\in(0, 1/2))$. Clearly,
$j'(p)=r^p (1+p \log r).$ An easy calculation shows that  $j'(p)=0 $ if and only if $p= -{1}/{\log r}$. Also, $j''\left(-{1}/{\log r}\right)={\log r}/{e}<0$. Thus,
\begin{equation}\label{eq12}
pr^p\leq j(-{1}/{\log r})=-{1}/{(e \log r)}\leq -{1}/{(e \log (1/2))}={1}/{(e \log 2)}.\end{equation}
Using equation \eqref{eq12}, for $r\in(0, 1/2)$, we have
\begin{align*}
Q'(r)&=-3+3 r^p+2 r^{2 p}+p r^{p-1} \left(-1+3 r-4 r^p+4 r^{1+p}\right)\\
&<-3+\frac{3}{2}+\frac{1}{2}-p r^{p-1} +3pr^p -4p r^{2p-1}+4p r^{2p}\\
&=-1+pr^p\left(-\frac{1}{r}+3\right)+4pr^{2p}\left(-\frac{1}{r}+1\right)\\
&<-1+pr^p-4pr^{2p}\\
&<-1+\frac{1}{e \log 2}-4pr^{2p}\\
&<0.
\end{align*}
Thus, $Q(r)=0$ has the unique root $r\in (0, 1/2)$, say $r \le r^*_{p}$.
Therefore, equation $(\ref{2.2.2})$ holds for $r \le r^*_{p}$ and $p \ge 1$.
\\
\noindent
The result is sharp for the function
\begin{equation*}
    f(z)=\frac{z}{(1-z)}=\sum_{n=1}^{\infty} z^n.
\end{equation*}
\end{proof}
\begin{remark}
For different values of $p$ the radius converges to $0.333333$ upto $6$ decimal places as evident from the table given below
\end{remark}
\begin{table}[h!]
    \begin{center}
		\begin{tabular}{ |c|c|c|c|c|c|c|c|c|c| }
			\hline
			S. No. & 1 & 2&3&4&5&6&7\\
			\hline
			p & 2 & 3&4&5&6&7&8\\
			\hline
			$r^*_p$ & 0.327553 & 0.332707 & 0.333265 & 0.333326 & 0.333332 & 0.333333 &0.333333\\
			\hline
		\end{tabular}
        \caption{\textbf{Calculation of the  radius $r^*_p$ for various values of $p$.}}
		\label{Table:2}
    \end{center}
\end{table}

\noindent
We evaluate Bohr-Rogosinski radius for the class $\mathcal{C}_2$ in next two theorems.
\begin{theorem}
Suppose that $f(z)=z + \sum_{n=2}^{\infty} a_n z^n \in \mathcal{C}_2$ and  $N\geq2$ is an integer. Then
\begin{equation} \label{2.2.3}
|f(z)|+ \sum_{n=N}^{\infty} |a_n z^n|  \le d(f(0), \partial{f(\mathbb{D})})
\end{equation}
for $|z| \le r^*_{1,N}$, where $r^*_{1,N}$ is the solution of
\[ 3 r + 2 r^N-1=0.\]
The result is sharp.
\end{theorem}
\begin{proof}
Let $f\in  \mathcal{C}_2$. Using
  coefficients bound  ($\ref{coef2}$) and growth estimate ($\ref{growth2}$), for $|z| \le r$, we have
\begin{equation}\label{eq13}
|f(z)|+ \sum_{n=N}^{\infty} |a_n z^n|  \le
\frac{r}{1-r}  + \sum_{n=N}^{\infty} r^{n}.
\end{equation}
\noindent
By \eqref{eq9} and \eqref{eq13}, the  inequality  holds
\begin{equation*}
|f(z)|+ \sum_{n=N}^{\infty} |a_n z^n|  \le  d(f(0), \partial{f(\mathbb{D})})
\end{equation*}
if and only if
\begin{align*}
  \frac{r}{1-r}  + \sum_{n=N}^{\infty} r^{n}&=\frac{r+r^N}{1-r} \le   \frac{1}{2}
\end{align*}
\text{or equivalently} $ 3 r + 2 r^N-1=0.$
Let us defined the function $R:[0, 1]\rightarrow 1$  such that
\begin{equation*}
R(r)= 3 r + 2 r^N-1.
\end{equation*}
Since $R(0)=-1$, $R(1)=4$ and $R'(r)=3+2 N r^{-1+N}>0$, using the  Intermediate Value Theorem, the equation $R(r)=0$ has the unique root $r\in(0, 1)$, say  $r^*_{1,N}$.
\\
\noindent
To prove the sharpness, we assume
\begin{equation*}
    f(z)=\frac{z}{(1-z)}=\sum_{n=1}^{\infty} z^n.
\end{equation*}
At $z=r=r^*_{1,N}$, we have
\begin{align*}
|f(z)|+ \sum_{n=N}^{\infty} |a_n z^n|&= \left| \frac{z}{1-z} \right|+ \sum_{n=N}^{\infty} |z|^n\\
&=\frac{r}{1-r}+ \sum_{n=N}^{\infty} r^n\\
&=\frac{r+r^N}{1-r}\\
&=\frac{1}{2}\\
&=d(f(0), \partial{f(\mathbb{D})}).
\end{align*}
Hence, the result is sharp.
\end{proof}
\begin{theorem}
Suppose that $f(z)=z + \sum_{n=2}^{\infty} a_n z^n \in \mathcal{C}_2$, then
\begin{equation} \label{2.2.4}
|f(z)|^2+ \sum_{n=N}^{\infty} |a_n z^n|  \le d(f(0), \partial{f(\mathbb{D})})
\end{equation}
for $|z| \le r^*_{2,N}$, where $r^*_{2,N}$ is the solution of
\begin{equation}\label{eq21}
1-2 r-r^2-2 r^N+2 r^{1+N}=0.
\end{equation}
The result is sharp.
\end{theorem}
\begin{proof}For $|z| \le r$,
the coefficient bounds  ($\ref{coef2}$) and growth estimate  ($\ref{growth2}$) yield
\begin{align*}
|f(z)|^2+ \sum_{n=N}^{\infty} |a_n z^n|&  \le \left( \frac{r}{1-r} \right)^2 + \sum_{n=N}^{\infty} r^{n}\nonumber\\
&=\frac{r^2+r^N-r^{1+N}}{(1-r)^2}.
\end{align*}
Also, by \eqref{eq9}
\[
d(f(0), \partial{f(\mathbb{D})}) \ge 1/2.
\]
Thus, ($\ref{2.2.4}$) holds if and only if
\begin{equation*}
 \frac{r^2+r^N-r^{1+N}}{(1-r)^2}   \le   \frac{1}{2}
\end{equation*}
equivalently
\[-1+2 r+r^2+2 r^N-2 r^{1+N}\leq0.\]
Let us define a function
\begin{equation*}
S(r)=-1+2 r+r^2+2 r^N-2 r^{1+N}\quad (r\in[0,1 )).
\end{equation*}
It is easy to see that
$S(0)=-1<0$,  $S(1/2)=1/4+1/2^N>0$ and
\begin{align*}
S'(r)&= \frac{2 \left(r+r^2+N r^N-r^{1+N}-N r^{1+N}\right)}{r}\\
&=\frac{2 \left(r+r^2+ r^N(N-(1+N) r)\right)}{r}\\
&>\frac{2 \left(r+r^2- r^N\right)}{r}\\
&=\frac{2 \left(r(1-r^{N-1})+r^2\right)}{r}\\
&>0.
\end{align*}
In view of the above inequalities and the Intermediate Value Theorem, there exists the unique  root $r^*_{2,N}\in(0, 1/2)$ of the equation \eqref{eq21} which also satisfy the inequality $(\ref{2.2.4})$ for $r \le r^*_{2,N}$.
\\
\noindent
The result is sharp for the function
\begin{equation*}
f(z)=\frac{z}{(1-z)}=\sum_{n=1}^{\infty} z^n.
\end{equation*}
\end{proof}

\section{certain Bohr radii for the class $\mathcal{C}_3$}
\noindent
We have sharp coefficient bounds and distortion theorems for the class $\mathcal{C}_3$.
\begin{lemma} {\cite{Sil1972}}\label{lm3}
If $f(z)= z + \sum_{n=2}^{\infty} a_n z^n \in \mathcal{C}_3$, then
\begin{equation}\label{coef3}
 |a_n| \le \frac{2}{3} + \frac{1}{3n^2}  .
\end{equation}
This result is sharp, with equality for
\begin{equation*}
    f(z)= \frac{2z}{3(1-z)} - \frac{1}{3} \int_{0}^{z} \frac{\log(1-\zeta)}{\zeta} d\zeta.
\end{equation*}
\end{lemma}
\begin{lemma} {\cite{Sil1972}}\label{lm4}
If $f \in \mathcal{C}_3$ then
\begin{equation}
\frac{2r}{3(1+r)}+\frac{1}{3} \int_{0}^{r} \frac{\log(1+t)}{t} dt
\le |f(z)|
\le \frac{2r}{3(1-r)}-\frac{1}{3} \int_{0}^{r} \frac{\log(1-t)}{t} dt
\label{growth3}
\end{equation}
and
\begin{equation}
\frac{2}{3(1+r)^2}+ \frac{\log(1+r)}{3r}
\le |f^{'}(z)| \le
\frac{2}{3(1-r)^2}- \frac{\log(1-r)}{3r}
\label{distortion3}
\end{equation}
for all $(0 < |z| \le r)$ and equality holds in all cases for the extremal function
\begin{equation*}
    f(z)= \frac{2z}{3(1-z)} - \frac{1}{3} \int_{0}^{z} \frac{\log(1-\zeta)}{\zeta} d\zeta.
\end{equation*}
\end{lemma}
\noindent
Using the above coefficient bound, growth and distortion results we obtain the following Bohr-type inequalities for the class $\mathcal{C}_3$.
\begin{theorem}\label{th2}
Suppose that $f(z)=z + \sum_{n=2}^{\infty} a_n z^n \in \mathcal{C}_3$, then
\begin{equation}\label{2.3.1}
|f(z)| + |f^{'}(z)||z| + \sum_{n=2}^{\infty} |a_n z^n| \le d( f(0), \partial{f(\mathbb{D})})
\end{equation}
for $|z| \le r_{31}$, where $r_{31}$ is the solution of
\begin{equation*}
\frac{2 (2-r^2) r}{3 (1-r)^2}-\frac{1}{3} \int_{0}^{r} \frac{\log(1-t)}{t} dt-\frac{1}{3} \log(1-r)+ \sum_{n=2}^{\infty} \frac{r^n}{3 n^2}-\left(\frac{1}{3}+\frac{\pi ^2}{36}\right)=0.
\end{equation*}
The radius is sharp.
\end{theorem}
\begin{proof}
We begin the proof by evaluating
\begin{equation*}
d(f(0), \partial{f(\mathbb{D})})= \liminf \limits_{|z| \rightarrow 1} |f(z)-f(0)|.
\end{equation*}
From equation ($\ref{growth3}$), we get
\begin{equation*}
d(f(0), \partial{f(\mathbb{D})}) \ge \lim \limits_{r \rightarrow 1} \left(\frac{2r}{3(1+r)}+\frac{1}{3} \int_{0}^{r} \frac{\log(1+t)}{t} dt \right)
\end{equation*}
which gives
\begin{equation}\label{eq16}
d(f(0), \partial{f(\mathbb{D})}) \ge \frac{1}{3} +\frac{1}{3} \int_{0}^{1} \frac{\log(1+t)}{t} dt=\frac{1}{3}+\frac{\pi ^2}{36}.
\end{equation}
From  equations $(\ref{coef3})$, $(\ref{growth3})$ and $(\ref{distortion3})$, for all $|z| \le r$, we have
\begin{align}\label{eq17}
|f(z)|& + |f^{'}(z)||z| + \sum_{n=2}^{\infty} |a_n| |z|^n\\
&\le \frac{2r}{3(1-r)}-\frac{1}{3} \int_{0}^{r} \frac{\log(1-t)}{t} dt+ \frac{2r}{3(1-r)^2}-\frac{1}{3} \log(1-r)+ \sum_{n=2}^{\infty} \left( \frac{2}{3} + \frac{1}{3 n^2} \right)  r^n\nonumber\\
&=\frac{2 (2-r^2) r}{3 (1-r)^2}-\frac{1}{3} \int_{0}^{r} \frac{\log(1-t)}{t} dt-\frac{1}{3} \log(1-r)+ \sum_{n=2}^{\infty} \frac{r^n}{3 n^2}.
\end{align}
In order to find the radius $r_{31}\in (0, 1)$ satisfying the inequality $(\ref{2.3.1})$, we define the function $P:[0, 1)\rightarrow\mathbb{R}$ by

\begin{equation*}
P(r)=\frac{2 (2-r^2) r}{3 (1-r)^2}-\frac{1}{3} \int_{0}^{r} \frac{\log(1-t)}{t} dt-\frac{1}{3} \log(1-r)+ \sum_{n=2}^{\infty} \frac{r^n}{3 n^2}-\left(\frac{1}{3}+\frac{\pi ^2}{36}\right)
\end{equation*}
An easy calculation shows that
\begin{equation}\label{eq14}
P(0)=-\left(\frac{1}{3}+\frac{\pi ^2}{36}\right)<0
\end{equation}
and
\begin{align}\label{eq15}
P\left(\frac{1}{2}\right)&=2-\frac{\pi ^2}{36}+\frac{\log2}{3}+\frac{1}{36} \left(-6+\pi ^2-6 \log 2^2\right)+\frac{1}{36} \left(\pi ^2-6 \log 2^2\right)\nonumber\\
&=\frac{1}{36} \left(66+\pi ^2-12 \log 2^2+\log4096\right)\nonumber\\
&\simeq2.17839>0.
\end{align}
Equations \eqref{eq14} and \eqref{eq15} together with Intermediate Value Theorem yield that there exists a root $r\in(0, 1/2)$ of the equation $P(r)=0.$
Also,
\begin{align*}
P'(r)&=\frac{2 \left(2+2 r-3 r^2+r^3\right)}{3 (1-r)^3}-\frac{\log(1-r)}{3r}+\frac{1}{3 (1-r)}-\frac{1}{3} \left(1+\frac{\log(1-r)}{r}\right)\\
&=\frac{1}{3} \left(\frac{4-r (-5+(8-3 r) r)}{(1-r)^3}-\frac{2 \log(1-r)}{r}\right)\\
&=\frac{r \left(4+5 r-8 r^2+3 r^3\right)-2 (1-r)^3 \log(1-r)}{3 (1-r)^3 r}\\
&=\frac{r \left(4+(1-r) r (5-3 r)\right)-2 (1-r)^3 \log(1-r)}{3 (1-r)^3 r}\\
&>0.
\end{align*}
Thus, there exists the  unique  root $r\in(0, 1)$ of the equation $P(r)=0$, say $r_{31}$. By equations \eqref{eq16} and \eqref{eq17}, for $|z|\leq r_{31}$, we have
\[|f(z)| + |f^{'}(z)||z| + \sum_{n=2}^{\infty} |a_n z^n| \le d( f(0), \partial{f(\mathbb{D})}).\]
To prove the sharpness, consider the function
\begin{equation*}
    f(z)= \frac{2z}{3(1-z)} - \frac{1}{3} \int_{0}^{z} \frac{\log(1-\zeta)}{\zeta} d\zeta.
\end{equation*}
Note that
\[-\frac{1}{3}\int_0^r \frac{\log(1-t)}{t} \, dt=\frac{r}{3}+\frac{r^2}{12}+\frac{r^3}{27}+\frac{r^4}{48}+\frac{r^5}{75}+\frac{r^6}{108}+\cdots\]
and \[-(1/3) \log(1 - r)=\frac{r}{3}+\frac{r^2}{6}+\frac{r^3}{9}+\frac{r^4}{12}+\frac{r^5}{15}+\frac{r^6}{18}+\cdots\]
So, at  $z=r=r_{31}$, we have
\begin{align*}
|f(z)|& + |f^{'}(z)||z| + \sum_{n=2}^{\infty} |a_n| |z|^n\\
&= \left|\frac{2r}{3(1-r)} - \frac{1}{3} \int_{0}^{r} \frac{\log(1-\zeta)}{\zeta} d\zeta \right| + |r| \left| \frac{2}{3(1-r)^2}-\frac{\log(1-r)}{3r} \right|\\
  &+   \sum_{n=2}^{\infty} \left| \left( \frac{2}{3}+ \frac{n^2}{3}\right) \right| |r|^n\\
&= \frac{2r}{3(1-r)} - \frac{1}{3} \int_{0}^{r} \frac{\log(1-\zeta)}{\zeta} d\zeta  +   \frac{2r}{3(1-r)^2}-\frac{\log(1-r)}{3r}  +   \sum_{n=2}^{\infty}  \left( \frac{2}{3}+ \frac{n^2}{3}\right) r^n\\
&=\frac{2 (2-r^2) r}{3 (1-r)^2}-\frac{1}{3} \int_{0}^{r} \frac{\log(1-t)}{t} dt-\frac{1}{3} \log(1-r)+ \sum_{n=2}^{\infty} \frac{r^n}{3 n^2}\\
&=\frac{1}{3}+\frac{\pi ^2}{36}\\
&=\frac{1}{3} +\frac{1}{3} \int_{0}^{1} \frac{\log(1+t)}{t} dt\\
&=d(f(0), \partial{f(\mathbb{D})}).
\end{align*}
Hence, the result is sharp.
\end{proof}
\noindent
Next we prove the following Bohr radius associated with powers of the coefficient $(|a_n|)$ for the class $\mathcal{C}_3$.
\begin{theorem}
Suppose that $f(z)=z + \sum_{n=2}^{\infty} a_n z^n \in \mathcal{C}_3$, then
\begin{equation} \label{2.3.2}
|z|+ \sum_{n=2}^{\infty} |a_n z^n| + \sum_{n=2}^{\infty} |a_n|^{p} |z|^{np}  \le d(f(0), \partial{f(\mathbb{D})})
\end{equation}
for $|z| \le R_{p}$ and $p \ge 1$, where $R_{p}\in(0, 1/2)$ is the root of the equation
\begin{equation}\label{eq20}
\frac{(3-r) r}{3 (1-r)}+\sum_{n=2}^{\infty} \frac{r^n}{3 n^2}  + \sum_{n=2}^{\infty} \left( \frac{2}{3} + \frac{1}{3 n^2} \right)^p  r^{np}- \frac{1}{36} \left(12+\pi ^2\right)=0.
\end{equation}
The result is sharp.
\end{theorem}
\begin{proof}
Proceeding as in Theorem \ref{th2}, we have
\begin{equation}\label{eq18}
 d(f(0), \partial{f(\mathbb{D})}) \ge  \frac{1}{36} \left(12+\pi ^2\right).
\end{equation}
It follows from  equations $(\ref{coef3})$ that for $|z| \le r$, we have
\begin{align}\label{eq19}
|z|+ \sum_{n=2}^{\infty} |a_n| |z|^n+ \sum_{n=2}^{\infty} |a_n|^{p} |z|^{np}
&\le r+ \sum_{n=2}^{\infty} \left( \frac{2}{3} + \frac{1}{3 n^2} \right)  r^n + \sum_{n=2}^{\infty} \left( \frac{2}{3} + \frac{1}{3 n^2} \right)^p  r^{np}\nonumber\\
&=\frac{(3-r) r}{3 (1-r)}+\sum_{n=2}^{\infty} \frac{r^n}{3 n^2}  + \sum_{n=2}^{\infty} \left( \frac{2}{3} + \frac{1}{3 n^2} \right)^p  r^{np}
\end{align}
In view of the equations \eqref{eq18} and \eqref{eq19}, the inequality ($\ref{2.3.2}$) holds if and only if
\begin{equation*}
\frac{(3-r) r}{3 (1-r)}+\sum_{n=2}^{\infty} \frac{r^n}{3 n^2}  + \sum_{n=2}^{\infty} \left( \frac{2}{3} + \frac{1}{3 n^2} \right)^p  r^{np}\le  \frac{1}{36} \left(12+\pi ^2\right).
\end{equation*}
Now, we define a function $Q:[0, 1)\rightarrow\mathbb{R}$ such that
\begin{equation*}
Q(r)=\frac{(3-r) r}{3 (1-r)}+\sum_{n=2}^{\infty} \frac{r^n}{3 n^2}  + \sum_{n=2}^{\infty} \left( \frac{2}{3} + \frac{1}{3 n^2} \right)^p  r^{np}-\frac{1}{36} \left(12+\pi ^2\right).
\end{equation*}
An easy calculation shows that
\[Q(0)=-\frac{1}{36} \left(12+\pi ^2\right)<0\]
and
\begin{align*}
Q\left(\frac{1}{2}\right)&=\frac{1}{3}-\frac{(\log2)^2}{6}+\sum _{n=2}^{\infty } \left(\frac{2}{3}+\frac{2^{-n}}{3 n^2}\right)^p\frac{1}{2^{n p}}\nonumber\\
&=\frac{1}{6}\left(2-(\log2)^2\right)+\sum _{n=2}^{\infty } \left(\frac{2}{3}+\frac{2^{-n}}{3 n^2}\right)^p\frac{1}{2^{n p}}\nonumber\\
&>0.
\end{align*}
Since $Q(0).Q(1/2)<0$, by Intermediate Value Theorem, there exists a root $r\in(0, 1/2)$ satisfying the equation $Q(r)=0$. Moreover,
\begin{align*}
Q'(r)&=\frac{3-2 r+r^2}{3 (1-r)^2}-\frac{r+\log(1-r)}{3 r}+\sum _{n=2}^{\infty } \left(\frac{2}{3}+\frac{1}{3 n^2}\right)^p {np} r^{{np}-1}>0.
\end{align*}
Thus, there exists the unique root, say $R_{p}\in(0, 1/2)$ satisfying the equation \eqref{eq20}.
In order to prove the sharpness, consider the function
\begin{equation*}
    f(z)= \frac{2z}{3(1-z)} - \frac{1}{3} \int_{0}^{z} \frac{\log(1-\zeta)}{\zeta} d\zeta.
\end{equation*}
At $z=r=R_{p}$, we have
\begin{align*}
|z|+ \sum_{n=2}^{\infty} |a_n z^n| + \sum_{n=2}^{\infty} |a_n|^{p} |z|^{np}
&=|z|+ \sum_{n=2}^{\infty} \left|\left( \frac{2}{3} + \frac{1}{3 n^2} \right) z^n\right| + \sum_{n=2}^{\infty} \left|\left( \frac{2}{3} + \frac{1}{3 n^2} \right)\right|^{p} |z|^{np}\\
&=r+ \sum_{n=2}^{\infty} \left( \frac{2}{3} + \frac{1}{3 n^2} \right) r^n + \sum_{n=2}^{\infty} \left( \frac{2}{3} + \frac{1}{3 n^2} \right)^{p} r^{np}\\
&= \frac{1}{3} + \frac{1}{3} \int_{0}^{1} \frac{\log(1-t)}{t} dt\\
&=d(f(0), \partial{f(\mathbb{D})}).
\end{align*}
Hence, the result is sharp.
\end{proof}
\noindent
We find Bohr-Rogosinski radius for the class $\mathcal{C}_3$ in next two theorems.
\begin{theorem}
Suppose that $f(z)=z + \sum_{n=2}^{\infty} a_n z^n \in \mathcal{C}_3$, then
\begin{equation} \label{2.3.3}
|f(z)|+ \sum_{n=N}^{\infty} |a_n z^n|  \le d(f(0), \partial{f(\mathbb{D})})
\end{equation}
for $|z| \le R_{1,N}$, where $R_{1,N}$ is the solution of
\begin{equation*}
 \sum_{n=N}^{\infty}   \frac{r^{n}}{3n^2} -\frac{1}{3} \int_{0}^{r} \frac{\log(1-t)}{t} dt- \frac{12+\pi ^2-36 r-\pi ^2 r-24 r^N}{36 (1-r)}=0.
\end{equation*}
The result is sharp.
\end{theorem}
\begin{proof}
Using Lemma \ref{lm3} and growth estimate  ($\ref{growth3}$), for $|z| \le r$, we get
\begin{equation*}
|f(z)|+ \sum_{n=N}^{\infty} |a_n z^n|  \le \frac{2r}{3(1-r)}-\frac{1}{3} \int_{0}^{r} \frac{\log(1-t)}{t} dt+ \sum_{n=N}^{\infty} \left( \frac{2}{3} + \frac{1}{3n^2} \right) r^{n}.
\end{equation*}
It follows from  the above inequality that the  equation ($\ref{2.3.3}$) holds if and only if
\begin{equation*}
\frac{2r}{3(1-r)}-\frac{1}{3} \int_{0}^{r} \frac{\log(1-t)}{t} dt+ \sum_{n=N}^{\infty} \left( \frac{2}{3} + \frac{1}{3n^2} \right) r^{n} \le  \frac{1}{3} +\frac{1}{3} \int_{0}^{1} \frac{\log(1+t)}{t} dt.
\end{equation*}
In order to find the radius $R_{1,N}\in (0, 1)$ satisfying the inequality $(\ref{2.3.3})$, we define the function $R:[0, 1)\rightarrow\mathbb{R}$ by
\begin{equation*}
R(r)=\sum_{n=N}^{\infty}   \frac{r^{n}}{3n^2} -\frac{1}{3} \int_{0}^{r} \frac{\log(1-t)}{t} dt- \frac{12+\pi ^2-36 r-\pi ^2 r-24 r^N}{36 (1-r)}.
\end{equation*}
It is easy to see that
$R(0)=-1/3-\pi^2/36<0$ and
\begin{align*}
R\left(\frac{1}{2}\right)&=\sum_{n=N}^{\infty}\frac{1}{3n^2}\left(\frac{1}{2^{n}}\right)+\frac{1}{36} \left(\pi ^2-6 (\log2)^2\right)+\frac{2}{3} \left(1+2^{1-N}\right)-\frac{1}{36} \left(12+\pi ^2\right)\\
&=\sum_{n=N}^{\infty}\frac{1}{3n^2}\left(\frac{1}{2^{n}}\right)+\frac{1}{6} \left(2+2^{3-N}-(\log 2)^2\right)\\
&>0.
\end{align*}
In view of the above and using Intermediate Value Theorem it follows that there exists a root $\varepsilon\in(0, 1/2)$ such that $R(\varepsilon)=0.$
Since \
\[
R'(r)=\frac{2}{3 (1-r)^2}-\frac{\log(1-r)}{3 r}+\frac{2 r^{N-1} (N+r-N r)}{3 (1-r)^2}+\sum_{n=N}^{\infty}   \frac{r^{n-1}}{3n}>0,
\] there exists the unique $ R_{1,N}\in (0, 1/2)$ such that $R(R_{1,N})=0.$
\\
\noindent
Thus, equation $(\ref{2.3.3})$ holds for $r \le R_{1,N}$, where $R_{1,N}$ is the solution of the equation
\begin{equation*}
\frac{2r}{3(1-r)}-\frac{1}{3} \int_{0}^{r} \frac{\log(1-t)}{t} dt+ \sum_{n=N}^{\infty} \left( \frac{2}{3} + \frac{1}{3n^2} \right) r^{n} -\frac{1}{3}- \frac{1}{3} \int_{0}^{1} \frac{\log(1+t)}{t} dt=0.
\end{equation*}
To prove the sharpness, we assume
\begin{equation*}
   f(z)= \frac{2z}{3(1-z)} - \frac{1}{3} \int_{0}^{z} \frac{\log(1-\zeta)}{\zeta} d\zeta.
\end{equation*}
At $z=r=R_{1,N}$, we have
\begin{align*}
|f(z)|+ \sum_{n=N}^{\infty} |a_n z^n|&=  \left|\frac{2z}{3(1-z)} - \frac{1}{3} \int_{0}^{z} \frac{\log(1-t)}{t} dt \right|+\sum_{n=N}^{\infty} \left|\left( \frac{2}{3} + \frac{1}{3 n^2} \right)\right||z|^{n}\\
 &=\frac{2r}{3(1-r)} - \frac{1}{3} \int_{0}^{r} \frac{\log(1-t)}{t} dt +\sum_{n=N}^{\infty} \left( \frac{2}{3} + \frac{1}{3 n^2} \right)r^{n}\\
 &=\frac{1}{3} + \frac{1}{3} \int_{0}^{1} \frac{\log(1-t)}{t} dt\\
 &=d(f(0), \partial{f(\mathbb{D})}).
\end{align*}
Hence, the result is sharp.
\end{proof}
\begin{theorem}
Suppose that $f(z)=z + \sum_{n=2}^{\infty} a_n z^n \in \mathcal{C}_3$, then
\begin{equation} \label{2.3.4}
|f(z)|^2+ \sum_{n=N}^{\infty} |a_n z^n|  \le d(f(0), \partial{f(\mathbb{D})})
\end{equation}
for $|z| \le R_{2,N}$, where $R_{2,N}$ is the solution of
\begin{equation*}
\left(
\frac{2r}{3(1-r)}-\frac{1}{3} \int_{0}^{r} \frac{\log(1-t)}{t} dt
\right)^2
+ \sum_{n=N}^{\infty}   \frac{r^{n}}{3n^2}  +\frac{ 24 r^N+(12 +\pi ^2)r-12-\pi ^2  }{36 (1-r)}=0.
\end{equation*}
The result is sharp.
\end{theorem}
\begin{proof}
It follows from Lemma \ref{lm3} and Lemma \ref{lm4} that
\begin{align*}
|f(z)|^2+ \sum_{n=N}^{\infty} |a_n z^n| & \le \left( \frac{2r}{3(1-r)}-\frac{1}{3} \int_{0}^{r} \frac{\log(1-t)}{t} dt \right)^2 + \sum_{n=N}^{\infty} \left( \frac{2}{3} + \frac{1}{3n^2} \right) r^{n}\\
&=\left( \frac{2r}{3(1-r)}-\frac{1}{3} \int_{0}^{r} \frac{\log(1-t)}{t} dt \right)^2 + \sum_{n=N}^{\infty} \frac{r^{n}}{3n^2}+\frac{2 r^N}{3(1-r)}
\end{align*}
for $|z| \le r$.
\\
\noindent
Define a function $S:[0, 1)\rightarrow\mathbb{R}$ by
\begin{equation*}
S(r)=\left( \frac{2r}{3(1-r)}-\frac{1}{3} \int_{0}^{r} \frac{\log(1-t)}{t} dt \right)^2+ \sum_{n=N}^{\infty} \left( \frac{2}{3} + \frac{1}{3n^2} \right) r^{n} -\frac{1}{3}- \frac{1}{3} \int_{0}^{1} \frac{\log(1+t)}{t} dt.
\end{equation*}
Note that
$S(0)=-{1}/{3}-{\pi ^2}/{36}$.
Also, an easy calculation yields
\begin{align*}
S(1/2)&=\left( \frac{2}{3}-\frac{1}{3} \int_{0}^{1/2} \frac{\log(1-t)}{t} dt \right)^2+ \sum_{n=N}^{\infty}  \frac{1}{3n^22^n}  +\frac{2^{2-N}}{3}-\frac{1}{36} \left(12+\pi ^2\right)\\
&=\frac{\left(24+\pi ^2-6 (\log2)^2\right)^2}{1296}+ \sum_{n=N}^{\infty}  \frac{1}{3n^22^n}  +\frac{2^{2-N}}{3}-\frac{1}{36} \left(12+\pi ^2\right)\\
&=\frac{\left(24+\pi ^2-6 (\log2)^2\right)^2}{1296}-\frac{1}{36} \left(12+\pi ^2\right)+ \sum_{n=N}^{\infty}  \frac{1}{3n^22^n}  +\frac{2^{2-N}}{3}\\
&=\frac{144+12 \pi ^2+\pi ^4-288 (\log2)^2-12 \pi ^2 (\log2)^2+36 (\log2)^4}{1296}+ \sum_{n=N}^{\infty}  \frac{1}{3n^22^n}  +\frac{2^{2-N}}{3}\\
&>0.
\end{align*}
In view of the above, $S(0)S(1/2)<0$, and by the  Intermediate Value Theorem there exists a root $\zeta\in(0, 1/2)$ such that $S(r)\leq 0$ for $r\leq \zeta$. Furthermore, Since
\begin{align*}
S'(r)&=2\left( \frac{2r}{3(1-r)}-\frac{1}{3} \int_{0}^{r} \frac{\log(1-t)}{t} dt \right)\left(\frac{2}{3 (1-r)^2}-\frac{\log(1-r)}{3 r}\right)\\
&+\frac{2 r^{-1+N} (N+r-N r)}{3 (1-r)^2}+\sum_{n=N}^{\infty}   \frac{r^{n-1}}{3n}\\
&>0,
\end{align*}
there exists the unique $ R_{2,N}\in(0, 1/2)$ such that the inequality $(\ref{2.3.4})$ holds for $r \le R_{2,N}$, and satisfying the equation $S( R_{2,N})=0.$
\\
\noindent
The result is sharp for the function
\begin{equation*}
   f(z)= \frac{2z}{3(1-z)} - \frac{1}{3} \int_{0}^{z} \frac{\log(1-\zeta)}{\zeta} d\zeta.
\end{equation*}
\end{proof}

\end{section}


\noindent
\textbf{Data availability statement} Data sharing not applicable to this article as no
datasets were generated or analysed during the current study.
%

%

\section*{Declarations}

%
\noindent
\textbf{Conflict of interest}  The authors declare they have no conflict of interest regarding the publication of this paper.

\noindent
\textbf{Funding}
No funding from any source.\\

\bibliographystyle{abbrv}
\bibliography{NJ-bibliography}

%
%
%
%
%

\end{document}